\documentclass[11pt]{amsart}

\usepackage{amsmath,amsthm,amssymb,mathtools}
\usepackage{fourier}
\usepackage{subcaption}
\usepackage{tikz,pgfplots}
\usepgfplotslibrary{fillbetween}
\usepackage{hyperref}
\usepackage{xcolor}
\usepackage{parskip}

\hypersetup{
  colorlinks=true,
  linktoc=all,
  linkcolor=blue,
  citecolor=red
}

\numberwithin{equation}{section}
\newtheorem{theorem}{Theorem}[section]

\newtheorem{lemma}[theorem]{Lemma}
\newtheorem{corollary}[theorem]{Corollary}
\newtheorem*{conjecture*}{Conjecture}

\theoremstyle{definition}

\newtheorem{example}[theorem]{Example}

\theoremstyle{remark}
\newtheorem{remark}[theorem]{Remark}

\newcommand{\R}{\mathbb{R}}
\newcommand{\N}{\mathbb{N}}

\newcommand{\e}{\varepsilon}
\newcommand{\J}{J}

\newcommand{\Int}{\mathrm{Int}}
\newcommand{\supp}{\mathrm{supp}}
\newcommand{\rank}{\mathrm{rank}}

\allowdisplaybreaks

\begin{document}
\sloppy
%\vspace*{-1cm}
\title{On Falconer type functions and the distance set problem}
\author{Minh-Quy Pham}
\address{Department of Mathematics, University of Rochester, USA.} \email{qpham3@ur.rochester.edu,quypham.math@gmail.com}

\begin{abstract} We study the distance set problem for pairs of compact sets $A, B\subset \mathbb{R}^n$, $n\geq 2$. We show that if $B$ is contained in a hyperplane and
\begin{align*}
    \dim_{H} A+\dim_{H} B>n,
\end{align*}
then the distance set $  \Delta(A,B)\coloneqq\left\{ \vert x-y\vert: x\in A, y\in B\right\}$ has positive Lebesgue measure, and the dimensional threshold is sharp. This yields new positive results for Falconer's distance problem in certain regimes, particularly where the best known bounds fail to apply.

We further establish Falconer's distance conjecture for certain classes of product sets under additional  structural assumptions. Specifically, if $A=A_1\times A_2\subset \mathbb{R}^{m}\times \mathbb{R}^{n-m}$ for some $0\leq m\leq n-1$, where $A_2$ is a Salem set, and
\begin{align*}
    \dim_HA>\frac{n}{2},
\end{align*}
then the distance set $\Delta(A):=\left\{|x-y|: x,y\in A\right\}
$ has positive Lebesgue measure. A key feature of our argument is the interpretation of the original map as a suitable projection. We extend the analysis to a broad class of smooth functions, recovering the sharp result of Koh, Pham, and Shen (\textit{J. Funct. Anal.} \textbf{286} (2024)) for quadratic polynomials in three variables.

\vspace{0.2cm}\noindent\textit{MSC 2020.} Primary: 42B10, 28A75; Secondary: 28A78, 28A80.\\
\textit{Key words and phrases:} Fourier transform, Falconer's distance problem, Hausdorff dimension, Fourier dimension.
\end{abstract}
\maketitle

\section{Introduction}\label{sec_1}
 Let $\Phi: \mathbb{R}^n\to \mathbb{R}$ be a smooth function, $n\geq 1$, and let $E\subset \mathbb{R}^n$ be a Borel set. We consider the following configuration question: 
 
 \hspace{0.5cm}How large does the Hausdorff dimension $\dim_{H} E$ need to be to guarantee that the set
\begin{align*}
    \Delta_\Phi(E)\coloneqq\left\{\, \Phi(x): x\in E\,\right\}
\end{align*}
has positive Lebesgue measure? 

In other words, we are interested in the case $0<\dim_{H} E<n$, and we wish to identify structural or regularity conditions on  $\Phi$ to ensure that for any set $E$ with a sufficiently large Hausdorff dimension, we get a positive result. Equivalently, we would like to determine the critical exponent for which the conclusion holds. We call $\Phi$ a Falconer type function if there exists a constant $0<\alpha<n$ such that for any compact set $E$ with $\dim_HE>\alpha$, one has $\mathcal{L}^1(\Delta_\Phi(E))>0$, where $\mathcal{L}^1$ is one-dimensional Lebesgue measure.
%, and if the Hausdorff dimension of the set is smaller than a certain threshold, then the conclusion fails. 
For a general function $\Phi:\mathbb{R}^n\to \mathbb{R}$, this property may fail, that is, there may exist a set $E$ such that $\dim_HE=n$, but $\mathcal{L}^1(\Delta_\Phi(E))=0$. For instance, considering the natural projection onto the first coordinate $\pi_1(x_1,\dots,x_n)=x_1$, one can find a compact set $A\subset \mathbb{R}$ such that $\dim_HA=1$ while $\mathcal{L}^1(A)=0$, and $\dim_H(A\times \mathbb{R}^{n-1})=n$, but $\pi_1(A\times \mathbb{R}^{n-1})=A$ still has zero Lebesgue measure.

In this paper, we study this problem in the Cartesian product setting, where $E$ has the form $E=A\times B\subset \mathbb{R}^m\times \mathbb{R}^{n-m}$, for some integers $1\leq m<n$. Our goal is to determine the conditions on $\dim_{H} A$, $\dim_{H} B$ to guarantee that
\begin{align*}
    \Delta_\Phi(A,B):=\left\{\Phi(x,y): x\in A, y\in B\right\}
\end{align*}
has positive Lebesgue measure.

The study of this question is motivated by Falconer's distance problem, a central question in geometric measure theory that has been extensively investigated for decades by many authors. In the early 1980s, Falconer \cite{Falconer85} considered the distance function $\Phi: \mathbb{R}^d\times \mathbb{R}^d\to \mathbb{R}$, $(x,y)\mapsto \vert x-y\vert$, where $d\geq 2$, and asked: How large does the Hausdorff dimension of a compact subset $E$ of $\mathbb{R}^d$ need to be to ensure that the distance set $\Delta(E)=\left\{ \vert x-y\vert: x,y \in E\right\}$ has positive Lebesgue measure? He showed that $\mathcal{L}^1(\Delta(E))>0$ whenever $\dim_{H} E>\frac{d+1}{2}$, and conjectured the following sharp dimensional threshold.
\begin{conjecture*}
    If $d\geq 2$ and $E\subset \mathbb{R}^d$ is a Borel set with $\dim_H E>\frac{d}{2}$, then $\mathcal{L}^1(\Delta(E))>0$.
\end{conjecture*}
This conjecture remains open in all dimensions $d\geq 2$, and the exponent $\frac{d}{2}$ is known to be sharp. The best currently known thresholds are given by
\begin{align}\label{eq_distance_best}
  \dim_HE>  \left\{
\begin{array}{lll}
    \tfrac{5}{4}, & d=2& (\text{Guth, Iosevich, Ou, and Wang \cite{GIOW}})\\
    \frac{d}{2}+\frac{1}{4}-\frac{1}{8d+4},&d\geq 3& (\text{Du, Ou, Ren, and Zhang \cite{Duetal23a,Duetal23b}})
\end{array}
   \right.
\end{align}
See \cite{Duetal21a, Duetal21b, Duetal23a, DuZhang19, GIOW} for other recent improvements on the dimensional thresholds in the distance problem. Various extensions and generalizations of this question have also been studied in a number of works; see, for example, \cite{Bennettetal16, Erdogan05, ESWARATHASAN2011, Grafakosetal15, AlexLiu16, Iosevichetal2025, Koh2024, Liu2020pinned_distance, Mattila_1987, PALSSON2023, Wolff99} and the references therein.

In a significant advance, Eswarathasan, Iosevich, and Taylor \cite{ESWARATHASAN2011} extended Falconer's theorem to a general class of metric functions. Specifically, when $\Phi: \mathbb{R}^d\times \mathbb{R}^d\to \mathbb{R}$ is a smooth function satisfying the Phong-Stein rotational curvature condition at $t\in \mathbb{R}$, namely,
\begin{align*}
    \det\begin{pmatrix}
        0&\nabla_x \Phi\\
        -(\nabla_y\Phi)^T&\frac{\partial^2\Phi}{\partial x_i\partial y_j}
    \end{pmatrix} \neq 0, \quad \forall (x,y)\in \Phi^{-1}(t),
\end{align*}
they showed that if $E\subset \mathbb{R}^d$ is a compact set with $\dim_{H} E>\frac{d+1}{2}$, then $\mathcal{L}^1(\Delta_\Phi(E,E))>0$ (see \cite[Theorem 1.8]{ESWARATHASAN2011}). Under the hypotheses that $n=2d\geq 4$ is even and the curvature condition on $\Phi$, the authors exploited the best known Sobolev bounds for the generalized Radon transforms to obtain the desired result. However, due to the generality of the class of functions $\Phi$, their result does not necessarily yield sharp dimensional thresholds in all cases. In addition, their method does not directly extend to functions of asymmetric form $\Phi:\mathbb{R}^m\times \mathbb{R}^{n-m}\to \mathbb{R}$, for $1\leq m\leq n$.

Recently, Koh, Pham, and Shen \cite{Koh2024} developed a discrete-to-continuous approach to study a class of Falconer type functions in three variables, with a key ingredient taken from the result of Eswarathasan, Iosevich, and Taylor in \cite{ESWARATHASAN2011}. More precisely, they considered quadratic polynomials depending on all variables that do not have the form $g(h(x)+k(y)+l(z))$. For such a polynomial $f$, and compact sets $A,B,C\subset \mathbb{R}$, they proved that if $\dim_{H} A+\dim_{H} B+\dim_{H} C>2$ then $\mathcal{L}^1\left(f(A,B,C)\right)>0$ \cite[Theorem 1.2]{Koh2024}. 
The dimensional threshold is optimal in general. As an application, they also gave an improvement to the distance problem for Cartesian product sets in three dimensions \cite[Corollary 1.4]{Koh2024}.

If one attempts to generalize the approach introduced in \cite{Koh2024}, a principal obstruction arises as follows. In \cite{Koh2024}, instead of studying the function $f\in \mathbb{R}[x,y,z]$ directly, the innovative step was to study the auxiliary map $\Psi(x,y,z,x', y',z')=f(x,y,z)-f(x',y',z')$. By applying a carefully chosen change of variables adapted to $\Psi$, the transformed function $\tilde\Psi:\mathbb{R}^3\times \mathbb{R}^3\to \mathbb{R}$ satisfies the Phong-Stein rotational curvature condition. This allowed them to use the framework developed in \cite{ESWARATHASAN2011} to obtain a positive result. However, the proof is technically intricate and highly specialized to the quadratic setting in three variables, as one must construct a specific change of variables adapted to the map $\Psi$, making generalization to higher degree multivariate polynomials and other classes of smooth functions highly nontrivial, and in some cases not even possible.

The first goal of this paper is to revisit the result on quadratic polynomials in three variables obtained in \cite{Koh2024} from a projection theoretic perspective, providing a simpler and more intuitive proof. Beyond this, we establish positive results for certain classes of higher degree multivariate polynomials and smooth functions. As will be shown in the next section, this approach naturally yields a direct application to the distance problem, which constitutes the second main goal of the paper. 

\subsection{Falconer type functions}
In this section, we discuss results for a class of Falconer type functions. The key idea of our approach is to reinterpret certain functions of interest as projection maps from suitable spaces onto lines. This formulation allows us to invoke well known projection theorems to establish the desired results. This approach bypasses the technical arguments in the previous methods and leads to a simpler proof with broad applicability.

Our first main result is stated below. Heuristically, suppose that $\Phi:\mathbb{R}^m\times \mathbb{R}^p\to \mathbb{R}$, for $1\leq m,p\leq n$, can be expressed as the dot product of two vector valued functions in $\mathbb{R}^n$, for some sufficiently large $n$. If these vector valued functions satisfy certain mild regularity conditions, then one obtains a positive result.
Throughout the rest of the paper, we denote $\dim_{H} A$ and $\dim_{F} A$ as the Hausdorff dimension and the Fourier dimension of $A$, respectively (see Section \ref{sec_2} for definitions). Let $\rho:\mathbb{R}^n\setminus \{0\}\to S^{n-1}$ denote the radial projection defined by $ \rho(x)=\frac{x}{\Vert x \Vert}$, for $x\in \mathbb{R}^n\setminus \{0\}$.

\begin{theorem}\label{thm_main}
    Let $1\leq m\leq n,  1\leq p\leq n-1$, let $U\subset \mathbb{R}^m$, and $V\subset \mathbb{R}^p$ be open sets. Let $P:U\to \mathbb{R}^{n}$ and $Q:V\to \mathbb{R}^{n}\setminus \{0\}$ be smooth maps such that $\rank DP$ and $\rank D(\rho\circ Q)$ are maximal everywhere. Let $\Phi:\mathbb{R}^m\times \mathbb{R}^p\to \mathbb{R}$ be a function defined by
    \begin{align*}
        \Phi(x,y)\coloneqq \sum\limits_{k=1}^{n}P_k(x)Q_k(y),\quad \forall (x,y)\in U \times V.
    \end{align*}
    Assume that $A\subset U$, $B\subset V $ are compact sets such that
        $\dim_{H}  A+\dim_{H}  B>n$.
    Then there exists $y\in B$ such that $\mathcal{L}^1\left(\Delta_{\Phi,y}(A)\right)>0$.

    Here, for $y\in V$ (or in $\R^p$), we denote $\Delta_{\Phi,y}(A)=:\left\{ \Phi(x,y): x\in A\right\}$.
\end{theorem}
We next describe our result for polynomials. In this setting, we are interested in the case where the function can be written as a polynomial in one variable, say $y$, whose coefficient functions, depending on the remaining variables, are regular. It turns out that the criterion for establishing a positive result is simply that the polynomial is not homogeneous in the variable $y$. The precise statement of the result is as follows. 
\begin{theorem}\label{thm_polys}
    Let $n\geq 2$, and let $f\in \mathbb{R}[x_1,\dots,x_{n},y]$ be a polynomial of the form
    \begin{align*}
        f(x,y)=ay^{r_{n+1}}+\sum\limits_{k=1}^{n}P_k(x)y^{r_k},
    \end{align*}
    where $a\in \R$, and $P:\mathbb{R}^{n}\to \mathbb{R}^{n}$ satisfies $\rank DP$ is maximal everywhere in some open set $U\subset \mathbb{R}^{n}$; and $r_1,r_2,\dots,r_{n+1}$ are nonnegative integers such that $r_i\neq r_j$, for some $1\leq i<j\leq n$. 
    Assume that $A\subset U$ and $B\subset \mathbb{R}$ are compact sets such that
    \begin{align*}
        \dim_{H} A+\dim_{H}  B>n.
    \end{align*}
    Then there exists $y\in B$ such that $f(A,y)=\left\{ f(x,y): x\in A\right\}$ has positive Lebesgue measure.
\end{theorem}

\begin{remark}
    \begin{itemize}
        \item[]  
        \item[(i)] The assumption that the polynomial $\sum\limits_{k=1}^{n}P_k(x)y^{r_k}$ is nonhomogeneous in the variable $y$ is necessary, see Example \ref{example_5.2} in Section \ref{sec_5}. This also tells us that the rank condition on $ D(\rho\circ Q)$ in Theorem \ref{thm_main} is necessary. 
        \item[(ii)] The lower bound for the sum of $\dim_HA, \dim_HB$ in Theorems \ref{thm_main} and \ref{thm_polys} is generally sharp (see Section \ref{sec_5}), and is directly in line with the main result in \cite{AralaChow24}, where they consider the family of regular (or "nice") polynomials and the sets of product type; see \cite[Theorem 1.2]{AralaChow24}. 
    \end{itemize}
\end{remark}

As a corollary of our general theorem, when $n=2$, we obtain the following result for quadratic polynomials in three variables. Algebraically, one can verify that the polynomials appearing in Theorem \ref{thm_polys_3vars} can be written in the form satisfying the hypotheses of Theorem \ref{thm_polys} (see \cite[Lemma 1.5]{AralaChow24}), hence the conclusion follows.
\begin{theorem}\label{thm_polys_3vars}
    Let $f\in \mathbb{R}[x,y,z]$ be a polynomial that depends on all variables and does not have the form $g(h(x)+k(y)+l(z))$. Suppose that $A,B,C\subset \R$ are compact sets.
    \begin{itemize}
        \item[(i)] If $\dim_HA+\dim_HB+\dim_HC>2$, then $\mathcal{L}^1(f(A,B,C))>0$.
        \item[(ii)] If $\dim_HA+\dim_HB+\dim_HC\leq 2$, then $\dim_Hf(A,B,C)\geq\frac{\dim_HA+\dim_HB+\dim_HC}{2}$. 
    \end{itemize}
\end{theorem}
\begin{remark}
    \begin{itemize}
        \item[]
        \item[(i)] Theorem \ref{thm_polys_3vars}(i) recovers the main result of \cite{Koh2024}. While the framework developed in \cite{Koh2024} is elegant and highlights a strong connection between discrete and continuous methods, their approach is effective only for sets of sufficiently large dimension. Theorem \ref{thm_polys_3vars}(ii) addresses the gap for sets of small dimension, extending a prime field result previously established in \cite{Pham_Vinh_deZeeuw_2019}. Specifically, the authors of \cite{Pham_Vinh_deZeeuw_2019} proved that for $A\subset \mathbb{F}_p$ with $|A|\le p^{2/3}$, one has $|f(A, A, A)|\gg |A|^{\frac{3}{2}}$. For further discussion, we refer the interested reader to the key references \cite{Pham_Vinh_deZeeuw_2019, Tao_expanding_poly_2015} and the references therein. 
        \item[(ii)] Recently, in \cite{Liao_Pham_Shen_2026}, Liao, Pham, and Shen proved an estimate for smoothed $L^2$ energy  for the natural measure supported on $\Delta_f(A,B,C)$.
    \end{itemize}
\end{remark}

\subsection{Distance sets}
In this section, we present applications of our results on Falconer type functions to a variant of Falconer's distance problem, namely the pinned distance problem. Given sets $A, B\subset \mathbb{R}^d$, we ask under which conditions there exists $y\in B$ such that the pinned distance set
\begin{align*}
    \Delta_y(A):=\left\{ \vert x-y\vert: x\in A\right\}
\end{align*}
has positive Lebesgue measure. This problem, originating in the work of Peres and Schlag \cite{PeresSchlag2000}, has been studied extensively in various settings.

We first consider the pinned distance problem in the case where the set of pins lies in an affine hyperplane.

\begin{theorem}\label{thm_distance_hyperplane}
Let $n\geq 2$,  and let $A$, $B$ be compact sets in $\mathbb{R}^n$. Assume that $B$ lies in a hyperplane and $$\dim_{H}  A+\dim_{H}  B>n.$$ Then there exists $y\in B$ such that $\mathcal{L}^1\left(\Delta_y(A)\right)>0$.

In particular, if there exists an affine hyperplane $W$ such that $\dim_{H} (A\cap W)> n-\dim_{H} A$, then there exists $y\in A$ such that $\mathcal{L}^1\left(\Delta_y(A)\right)>0$.
\end{theorem}
We next discuss some remarks on the assumptions and conclusions of the theorem.
\begin{remark}
\begin{itemize}
    \item[]
    \item[(i)] The dimensional threshold in Theorem \ref{thm_distance_hyperplane} is the best possible, see examples in Section \ref{sec_5}.
    \item[(ii)] In \cite{Mattila_1987}, Mattila showed that if $\dim_{H} A+\dim_{H} B>n$ and $\dim_{H} B\leq \frac{n-1}{2}$, then $\mathcal{L}^1\left(\Delta(A,B)\right)>0$, see \cite[Corollary 5.4]{Mattila_1987}. His result applies to general sets $B$, without assuming that $B$ is contained in an affine hyperplane. In contrast, Theorem \ref{thm_distance_hyperplane} does not impose any restriction on the dimension of $B$, and gives a positive result for the pinned distance set.
    \item[(iii)] Note that Theorem \ref{thm_distance_hyperplane} also gives a proof of Falconer's theorem. Indeed, if $A\subset \mathbb{R}^n$ has $\dim_HA>\frac{n+1}{2}$, $n\geq 2$, then by the results on the dimension of the plane sections of the sets, there exists a hyperplane $W$ such that $\dim_H(A\cap W)>\dim_HA-1$ (see Marstrand \cite{Marstrand54}, Mattila \cite{Mattila75}, or \cite[Theorem 6.8]{Mattila2015}). This implies
    \begin{align*}
        \dim_H(A\cap W)>\dim_HA-1>n-\dim_HA.
    \end{align*}
    Hence, by Theorem \ref{thm_distance_hyperplane}, one has $\mathcal{L}^1\left(\Delta(A)\right)>0$.
    \item[(iv)]Finally, note that the finite field counterpart of Theorem \ref{thm_distance_hyperplane} was proved by Kang, Koh, and Rakhmonov \cite{HunSeokDoowonFirdavs}. For general sets, it has been established in \cite{Murphyetal_2022} that for any $A\subset \mathbb{F}_p^2$ with $|A|\gg p^{\frac{5}{4}}$, the number of pinned distinct distances is at least $p/2$. This landmark result is parallel to the work of Guth, Iosevich, Ou, and Wang \cite{GIOW} on the Falconer distance problem, as mentioned earlier. The most recent progress over finite fields can be found in \cite{PhamYoo}, where the authors established a rotational version showing that for any sets $A, B \subset \mathbb{F}_p^2$ with $|A|, |B|\geq p$, for almost every rotation $g\in SO(2)$, one has $|\Delta(A,gB)|\geq p$. In light of the results in this paper, it is natural to ask whether analogous results would also hold over finite fields, using the theorems on projections developed in \cite{BrightLundPham, LundPhamThu, LundPhamVinh}.
\end{itemize}
\end{remark}
 Now assume that $A=A_1\times A_2\times \dots \times A_n\subset \mathbb{R}^n$ is a compact Cartesian product set such that $\dim_HA_n=\min\limits_{1\leq i\leq n}\dim_HA_i$. Without loss of generality, we can assume that $0\in A_n$. Putting $B=A_1\times\dots\times A_{n-1}\times\{0\}$, then $B$ is a subset of $A$ and lies in the hyperplane $\{x_n=0\}$. Applying Theorem \ref{thm_distance_hyperplane}, we obtain $\mathcal{L}^1\left(\Delta_y(A)\right)>0$ for some $y\in B$, provided that
  \begin{align*}
        \dim_{H} A>\frac{1}{2}\left(n+\min\limits_{i=1,\dots,n}\dim_{H} A_i\right).
    \end{align*}
This is an immediate corollary of Theorem \ref{thm_distance_hyperplane}, which gives a positive result for the distance set problem for a class of Cartesian product sets to which the best known results \eqref{eq_distance_best} do not apply.

\begin{corollary}\label{cor_2}
    Let $A=A_1\times A_2\times \dots \times A_n\subset \mathbb{R}^n$ be a compact Cartesian product set, where $A_i\subset \mathbb{R}$, for $1\leq i\leq n$. If
    \begin{align*}
        \dim_{H} A>\frac{1}{2}\left(n+\min\limits_{i=1,\dots,n}\dim_{H} A_i\right),
    \end{align*}
    then there exists $y\in A$ such that $\mathcal{L}^1\left(\Delta_y(A)\right)>0$.
\end{corollary}

\newcommand{\thresholdplota}{%
\begin{tikzpicture}
\begin{axis}[
  axis lines=middle, axis equal image,
  scale only axis, trim axis left, trim axis right,
  xlabel={$\alpha$},
  ylabel={$\dim_{H}(A)$},
  xlabel style={font=\footnotesize,at={(axis description cs:1,0)}, anchor=north},
  ylabel style={font=\footnotesize,at={(ticklabel* cs:1.1)}, anchor=east},
  clip=false,
  tick label style={font=\footnotesize},
  xmin=-0.1, xmax=1.1,
  ymin=0.5,  ymax=1.45,
  width=5.5cm, height=5.5cm,
  ytick={1,1.25}, yticklabels={$1$, $\tfrac{5}{4}$},
  xtick={0,0.5,1}, xticklabels={$0$, $\tfrac{1}{2}$, $1$}
]
  % lines
 \addplot[domain=0:0.5, very thick, blue] {1 + 0.5*x};
    \addplot[domain=0.5:1, very thick,dashed, blue] {1 + 0.5*x};
  \addplot[domain=0:0.5, thick, dashed] {5/4};
  \addplot[domain=0.5:1., thick] {5/4};

  % labels
  \node[anchor=south west, font=\footnotesize, text=blue]
    at (axis cs:0.45,1.45) {Corollary~\ref{cor_2}};
  \node[anchor=south west, font=\footnotesize]
    at (axis cs:0.8,1.25) {GIOW~\cite{GIOW}};
    \node[anchor=north, font=\small]
    at (axis cs:0,0.48) {$0$};

  % shading
  \addplot[name path=A, domain=0:1., draw=none] {5/4};
  \addplot[name path=B, domain=0:1., draw=none] {1 + 0.5*x};
  \path[name path=bottom] (axis cs:0,0.502) -- (axis cs:1,0.502);
  \addplot[teal!15] fill between[of=bottom and B, soft clip={domain=0:0.5}];
  \addplot[teal!15] fill between[of=bottom and A, soft clip={domain=0.5:1}];

  % vertical line + dot
  \addplot[black, dashed, thick] coordinates {(0.5,0.5) (0.5,1.25)};
  \addplot[only marks, mark=*] coordinates {(0.5,1.25)};
\end{axis}
\end{tikzpicture}%
}

\newcommand{\thresholdplotb}{%
\begin{tikzpicture}
\begin{axis}[
  axis lines=middle, axis equal image,
  scale only axis, trim axis left, trim axis right,
  xlabel={$\alpha$},
  ylabel={$\dim_{H}(A)$},
  xlabel style={font=\footnotesize,at={(axis description cs:1.0,0)}, anchor=north},
  ylabel style={font=\footnotesize,at={(ticklabel* cs:1.1)}, anchor=east},
  clip=false,
  tick label style={font=\footnotesize},
  xmin=-0.1, xmax=1.1,
  ymin=0.5,  ymax=1.45,
  width=5.5cm, height=5.5cm,
  ytick={1,1.25},
  yticklabels={$\tfrac{n}{2}$,$\gamma_n$},
  xtick={0,0.5,1},
  xticklabels={$0$, $\tfrac{1}{2}-\tfrac{1}{4n+2}$, $1$}
]
  % lines
  \addplot[domain=0:0.5, very thick, blue] {1 + 0.5*x};
    \addplot[domain=0.5:1, very thick,dashed, blue] {1 + 0.5*x};
  \addplot[domain=0:0.5, thick, dashed] {5/4};
  \addplot[domain=0.5:1., thick] {5/4};

  % labels
  \node[anchor=south west, font=\footnotesize, text=blue]
    at (axis cs:0.45,1.45) {Corollary~\ref{cor_2}};
  \node[anchor=south west, font=\footnotesize]
    at (axis cs:0.8,1.25) {DORZ \cite{Duetal23a,Duetal23b}};
    \node[anchor=north, font=\small]
    at (axis cs:0,0.48) {$0$};

  % shading
  \addplot[name path=A, domain=0:1., draw=none] {5/4};
  \addplot[name path=B, domain=0:1., draw=none] {1 + 0.5*x};
  \path[name path=bottom] (axis cs:0,0.502) -- (axis cs:1,0.502);
  \addplot[teal!15] fill between[of=bottom and B, soft clip={domain=0:0.5}];
  \addplot[teal!15] fill between[of=bottom and A, soft clip={domain=0.5:1}];

  % vertical line + dot
  \addplot[black, dashed, thick] coordinates {(0.5,0.5) (0.5,1.25)};
  \addplot[only marks, mark=*] coordinates {(0.5,1.25)};
\end{axis}
\end{tikzpicture}%
}

\begin{figure}[ht]
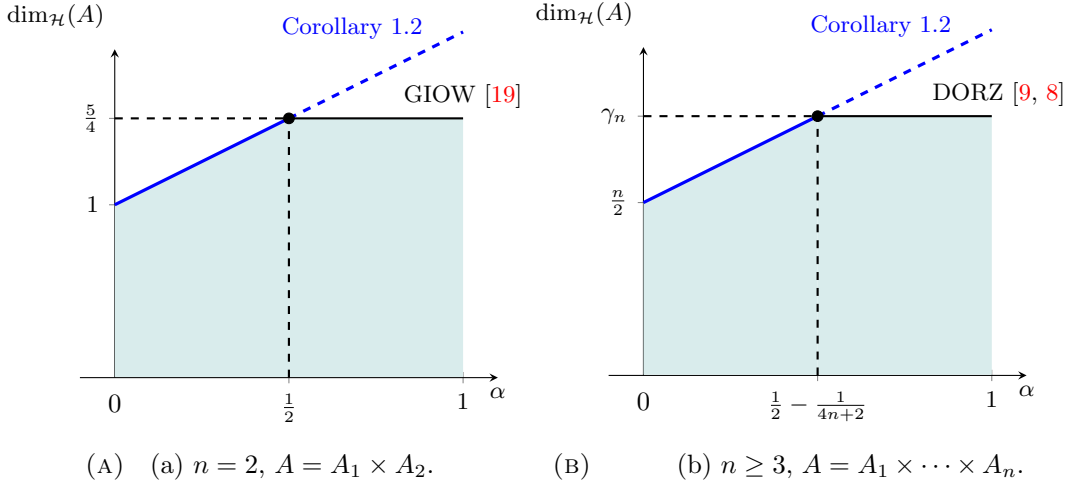

\centering
\vspace{-0.5em}

\begin{subfigure}[t]{0.46\linewidth}
    \centering
    \thresholdplota
    \captionsetup{width=0.9\linewidth}
    \caption{\hspace{0.2cm} $n = 2$, $A = A_1 \times A_2$.}
\end{subfigure}
\hspace{-0.2cm}
\begin{subfigure}[t]{0.46\linewidth}
    \centering
    \thresholdplotb
    \captionsetup{width=0.9\linewidth, justification=centering}
    \caption{\hspace{1.0cm} $n \ge 3$, $A = A_1 \times \cdots \times A_n$.}
\end{subfigure}

\vspace{-0.3em}  
\caption{\small 
Dimensional thresholds for the distance set problem for Cartesian product sets 
$A = A_1 \times \cdots \times A_n$, 
where
$\alpha =\min\limits_{i=1,\ldots,n} \dim_{H} A_i$, $\gamma_2=\tfrac{5}{4}$, and $\gamma_n=\tfrac{n}{2}+\tfrac{1}{4}-\tfrac{1}{8n+4},\quad \forall n\geq 3$.}
\label{fig:thresholds}
\vspace{-0.3em}  % tighten spacing below figure
\end{figure}

\begin{remark}
\begin{itemize}
    \item[] 
    \item[(i)]
     Let $n=2$, we can construct a compact set $A=A_1\times A_2$, where $A_1,A_2\subset \mathbb{R}$ such that 
    $\dim_{H} A_2<\frac{1}{2}$, and 
    \begin{align*}
        \frac{5}{4}> \dim_{H} A>1+\frac{1}{2}\dim_{H} A_2.
    \end{align*}
    Applying Corollary \ref{cor_2} yields a positive result, whereas the exponents in \eqref{eq_distance_best} do not, see Figure \ref{fig:thresholds} (A).
    
    Now, let $n\geq 3$, and let $A$ be as in Corollary \ref{cor_2}. Note that one can find compact sets $A_i\subset \mathbb{R} $, for $i=1,\dots,n$ such that
    \begin{align*}
      \frac{n}{2}+\frac{1}{4}-\frac{1}{8n+4}>\dim_{H} A>\frac{1}{2}\left(n+\min\limits_{i=1,\dots,n}\dim_{H} A_i\right),
    \end{align*}
    provided $\min\limits_{i=1,\dots,n}\dim_{H} A_i<\frac{1}{2}-\frac{1}{4n+2}$. Again, by Corollary \ref{cor_2}, we get a positive result, while \eqref{eq_distance_best} fails to apply, see Figure \ref{fig:thresholds} (B).
    \item[(ii)] Let $C\subset \mathbb{R}$ be a compact set, put $A\coloneqq C\times \dots\times C\subset \mathbb{R}^n$. By Corollary \ref{cor_2}, if $\dim_{H} C>\frac{n}{2n-1}$, then $\Delta(A)$ has positive Lebesgue measure. This recovers the result by Iosevich and Liu \cite{AlexLiu16} for Cartesian products $A=C\times \dots \times C$. 
    \end{itemize}
\end{remark}
Next, we turn to the case where both $A$ and $B$ are Cartesian product sets in $\mathbb{R}^{m}\times \mathbb{R}^{n-m}$, for some $0\leq m\leq n-1$. In this setting, we obtain a result similar to that of Theorem \ref{thm_distance_hyperplane}, but with additional assumptions on the structure and dimensions of the sets. 
\begin{theorem}\label{thm_distance_product}
    Let $n\geq 2$ and $0\leq m\leq n-1$. Let $A=A_1\times A_2$ and $B=B_1\times B_2$ be compact sets in $\R^n= \R^m\times \R^{n-m}$. Assume that 
    \begin{align*}
    \dim_{H} A_1+\dim_{H} A_2+\dim_{H} B_1+\dim_{F} B_2>n.
    \end{align*} 
    Then there exists $y\in B$ such that $\mathcal{L}^1\left(\Delta_y(A)\right)>0$.
\end{theorem}
As a special case of Theorem \ref{thm_distance_product} with $A=B$, we obtain the following corollary. In particular, the corollary shows that Falconer's distance conjecture holds for the class of product sets $A=A_1\times A_2\subset\R^m\times \R^{n-m}$ where $A_2$ is a Salem set.
Recall that a subset $A\subset \mathbb{R}^n$ is called a \textit{Salem set} if $\dim_HA=\dim_FA$ (see Section \ref{sec_2}).
\begin{corollary}\label{cor_3}
    Let $n\geq 2$ and $0\leq m\leq n-1$. Let $A$ be a compact set in $\mathbb{R}^n$ of the form $A=A_1\times A_2\subset \R^m\times \R^{n-m}$.
\begin{itemize}
    \item[(i)] If $\dim_{H} A_1+\frac{1}{2}\left(\dim_{H} A_2+\dim_{F} A_2\right)>\frac{n}{2}$, then there exists $y\in A$ such that $\mathcal{L}^1\left(\Delta_y(A)\right)>0$.
    \item[(ii)] If $A_2$ is a Salem set and
    \begin{align*}
        \dim_{H}A>\frac{n}{2},
    \end{align*}
    then there exists $y\in A$ such that $\mathcal{L}^1\left(\Delta_y(A)\right)>0$.
\end{itemize}
\end{corollary}
\begin{remark}
    \begin{itemize}
        \item[]
        \item[(i)] Theorem \ref{thm_distance_product} gives a positive result for certain classes of product sets where Theorem \ref{thm_distance_hyperplane} does not apply, as Theorem \ref{thm_distance_product} does not require that $B$ is contained in an affine hyperplane.
        \item[(ii)] In \cite[Theorem 5.3]{Mattila_1987}, Mattila proved that if $\dim_{H} A+\dim_{F} B>n$, then $\mathcal{L}^1\left(\Delta(A,B)\right)>0$. In particular, if $A=B$ is a Salem set, this reduces to the condition $\dim_{H} A>\frac{n}{2}$. In fact, Mattila's theorem was stated in the more general context of spherical Fourier dimensions and generalized Salem sets. 
        \\
        In contrast, Theorem \ref{thm_distance_product} applies to sets $B$ that are not necessarily Salem. This is possible because $\dim_{F} B_2$ may be equal to $0$, which implies that $\dim_{F} B$ is significantly smaller than $\dim_{H} B$.
        In particular, the result in \cite{Mattila_1987} does not apply to part (ii) of Corollary \ref{cor_3}.
    \end{itemize}
\end{remark}
The rest of the paper is organized as follows. In Section \ref{sec_2}, we recall preliminary notions and results regarding the dimensions of sets, including the definitions of Hausdorff dimension, Fourier dimension, Sobolev dimension, and Fourier spectrum. We also provide some well known estimates for dimensions of exceptional sets of projections. In Section \ref{sec_3}, we state and prove our general theorem for Falconer type functions using a projection-theoretic framework, Theorem \ref{thm_main_general}, from which Theorems \ref{thm_main}, \ref{thm_polys}, and \ref{thm_polys_3vars} follow. Section \ref{sec_4} is devoted to applications of the results in Section \ref{sec_3} to distance set problems. More precisely, we prove Theorem \ref{thm_distance_hyperplane} by considering the pinned distance problem for pins lying in a hyperplane and reducing it to the framework of the results in Section \ref{sec_3}. Theorem \ref{thm_distance_product} is then proved by combining the projection-theoretic approach with estimates on the Fourier spectrum of difference sets. Finally, in Section \ref{sec_5}, we present several examples demonstrating the sharpness of the dimensional thresholds in our main results.

\addtocontents{toc}{\protect\setcounter{tocdepth}{1}}
\section{Preliminaries}\label{sec_2}
In this section, we give basic notions and results that will be used in our proofs. Throughout the paper, we denote by $\mathcal{L}^n$ the Lebesgue measure in the Euclidean space $\mathbb{R}^n$, with $n\geq 1$.
Let $A\subset \mathbb{R}^n$, we denote by $\mathcal{M}(A)$ the set of non-zero Radon measures $\mu$ on $\mathbb{R}^n$ with compact support $\supp \mu \subset A$. 
%The Hausdorff dimension and Fourier dimension of $A$ will be denoted by $\dim_{H} A$ and $\dim_{F}A$, respectively. 
The Fourier transform of $\mu$ is defined by
\begin{align*}
    \widehat{\mu}(\xi)\coloneqq\int e^{-2\pi i \xi\cdot x}\,d\mu(x), \quad \xi\in \mathbb{R}^n.
\end{align*}

\subsection{Dimensions of sets}
We recall some notions regarding the dimensions of sets. The following lemma gives a definition for the Hausdorff dimension.
\begin{lemma}[Frostman's lemma, Theorem 2.7, \cite{Mattila2015}]
Let $0\leq s\leq n$.
For a Borel set $A\subset \mathbb{R}^n$, the $s-$dimensional Hausdorff measure of $A$ is positive if and only if there exists a measure $\mu\in\mathcal{M}(A)$ satisfying
\begin{align}\label{eq_frostman}
    \mu(B(x,r))\lesssim r^s, \quad \forall\, x\in \mathbb{R}^n,\,\,r>0.
\end{align}
In particular,
\begin{align*}
    \dim_{H} A
    &\coloneqq \sup\left\{ s\leq n: \exists \mu\in \mathcal{M}(A) \text{ such that } \eqref{eq_frostman} \text{ holds}\right\}.
\end{align*}
\end{lemma}
Frostman's lemma yields that for any exponent $0<s< \dim_{H} (A)$, there exists a probability measure $\mu$ on $A$ such that 
\begin{align}\label{eq_Frostman_measure}
    \mu(B(x,r))\lesssim r^{s}, \quad \forall\, x\in \mathbb{R}^n, r>0.
\end{align}
The $s$-energy integral of a measure $\mu\in \mathcal{M}(\mathbb{R}^n)$ (see \cite{Mattila2015})
is
\begin{align*}
    I_s(\mu)\coloneqq\iint\vert x-y\vert^{-s}\,d\mu(x)\,d\mu(y) = c(n,s)\int\vert\widehat{\mu}(\xi)\vert^2\vert\xi\vert^{s-n}\,d\xi.
\end{align*}
If $\mu\in \mathcal{M}(\mathbb{R}^n)$ satisfies the Frostman condition $\eqref{eq_frostman}$, then $I_t(\mu)<\infty$ for all $t\in (0,s)$. Thus, we have the following equivalent formulation for the Hausdorff dimension of a Borel set $A\subset \mathbb{R}^n$,
\begin{align*}
    \dim_{H}A
    &=\sup\left\{ s\leq n: \exists \mu\in \mathcal{M}(A) \text{ such that } I_s(\mu)<\infty\right\}.
\end{align*}
The Sobolev dimension of a measure $\mu\in \mathcal{M}(\mathbb{R}^n)$ is
\begin{align*}
    \dim_{S}\mu\coloneqq\sup\left\{ s\in \mathbb{R}:  \int \vert\widehat{\mu}(\xi)\vert^2(1+\vert\xi\vert)^{s-n}\,d\xi<\infty\right\}.
\end{align*}
A set $A\subset \mathbb{R}^n$ is said to have Sobolev dimension $\dim_SA\geq s$ if $A$ carries a Borel probability measure $\mu$ such that $\dim_S\mu\geq s$.

The Fourier dimension of a set $A\subset \mathbb{R}^n$ is
\begin{align*}
    \dim_{F}A\coloneqq\sup\left\{ s\leq n: \exists \mu\in  \mathcal{M}(A) \text{ such that } \vert\widehat{\mu}(\xi)\vert \lesssim (1+\vert\xi\vert)^{-\frac{s}{2}}, \text{ for all }\xi\in \mathbb{R}^n\right\}.
\end{align*}
Note that $0\leq \dim_{F} A\leq \dim_{H} A$, and the inequality can be strict. For example, subsets of hyperplanes have zero Fourier dimension, see \cite[Section 3.6]{Mattila2015}.

Given a measure $\mu\in \mathcal{M}(\mathbb{R}^n)$, we see that the Fourier and Sobolev dimensions quantify the decay rates of $\vert\widehat{\mu}(\xi)\vert$ by weighted $L^\infty$ and $L^2$ norms, respectively. In \cite{Fraser24}, Fraser introduced interpolated dimensions between the Fourier dimension and the Sobolev dimension of the measure, called the Fourier spectrum.

Let $\theta\in (0,1]$ and $s\in \mathbb{R}$, we define the $(s,\theta)$-energy of $\mu$ by
\begin{align*}
    \J_{s,\theta}(\mu)\coloneqq \left( \int_{\mathbb{R}^n}\vert\widehat{\mu}(\xi)\vert^{\frac{2}{\theta}}\vert\xi\vert^{\frac{s}{\theta}-n}\,d\xi\right)^\theta,
\end{align*}
and for $\theta=0$,
\begin{align*}
    \J_{s,0}(\mu)\coloneqq\sup_{\xi\in \mathbb{R}^n}\vert\widehat{\mu}(\xi)\vert^2\vert\xi\vert^s.
\end{align*}
Then the Fourier spectrum of $\mu$ at $\theta$ is
\begin{align*}
    \dim_{F}^\theta\mu\coloneqq\sup\left\{ s\in \mathbb{R}: \J_{s,\theta}(\mu)<\infty\right\}.
\end{align*}
For each $\theta\in [0,1]$, we have $\dim_{F}\mu\leq \dim_{F}^\theta\mu\leq \dim_S\mu$, with equality on the left if $\theta=0$, and equality on the right if $\theta=1$.
In \cite{Fraser24}, the author proved that $\dim_{F}^\theta\mu$ is non-decreasing, concave, and continuous for all $\theta\in [0,1]$ for compactly supported measures. 

Let $A\subset \mathbb{R}^n$ be a Borel set. We define the Fourier spectrum of $A$ at $\theta$ as
\begin{align*}
    \dim_{F}^\theta A\coloneqq\sup\left\{ \min\left\{\dim_{F}^\theta\mu,n\right\}: \mu \in \mathcal{M}(A)\right\}.
\end{align*}
Similarly, for all $\theta\in [0,1]$, we have $\dim_{F} A\leq \dim_{F}^\theta A\leq \dim_{H}  A$, with equality on the left if $\theta=0$ and equality on the right if $\theta=1$. Also, $\dim_{F}^\theta A$ is continuous for all  $\theta\in [0,1]$, see \cite[Theorem 1.5]{Fraser24}.

Some applications have been studied using the Fourier spectrum: distance problem, sumset type problems \cite{Fraser24}, orthogonal projections \cite{FraOre24}, and $L^2$ restriction problem \cite{CarFraOre24}.

\subsection{Exceptional sets of projections onto lines}
This section collects projection results needed in the proofs, in particular the estimates for the dimensions of exceptional sets of projections onto lines (see \cite[Section 5]{Mattila2015}). For $\theta\in S^{n-1}, n\geq 2$, define 
\begin{align*}
    \pi_\theta:\mathbb{R}^n\to \mathbb{R},\quad \pi_\theta(x)=\theta\cdot x,
\end{align*}
the orthogonal projection onto the line $l_\theta=\{t\theta: t\in \R\}$.

\begin{theorem}{\cite[Corollary 5.7]{Mattila2015}}\label{thm_projection_1}
Let $A \subset \mathbb{R}^n$, $n \geq 2$, be a Borel set.
\begin{itemize}
  \item[(i)] If $\dim_{H} A > 1$, then
  \[
  \dim_{H}  \left\{ \theta \in S^{n-1} : \mathcal{L}^1(\pi_\theta(A)) = 0 \right\} \leq n - \dim_{H} A.
  \]
  \item[(ii)] If $\dim_{H} A > 2$, then
\[
\dim_{H}  \left\{ \theta \in S^{n-1} : \text{the interior of } \pi_\theta(A) \text{ is empty} \right\} \leq n + 1 - \dim_{H} A.
\] 
\end{itemize}
\end{theorem}
Theorem \ref{thm_projection_1} part (i) was proven by Falconer \cite{Falconer_1982}, and part (ii) was proven by Peres and Schlag \cite{PeresSchlag2000}.
The upper bound in (i) is sharp (see \cite[Example 5.13]{Mattila2015}).

Recently, Cholak et al. \cite{Cholaketal24} proved sharp bounds for the dimension of the exceptional sets of projections for which the images have dimensions less than $1$. This upper bound was proved using induction on dimension $n$, and the base case $n=2$ was proved by Ren and Wang \cite{RenWang}.
\begin{theorem}{\cite[Theorem 1.2]{Cholaketal24}}\label{thm_projection_2}
    Let $A\subset \mathbb{R}^n$, $n\geq 2$, be a Borel set with $\dim_{H} A=a$. For $\max\{0,a-(n-1)\}\leq s\leq \min\{a,1\}$, we have
    \[
  \dim_{H}  \left\{ \theta \in S^{n-1} : \dim_{H} \pi_\theta(A)<s \right\} \leq S(a,s),
  \]
  where $S(a,s)\coloneqq n - 2 - \lfloor a - s \rfloor + \max\left\{\,0,\, \lfloor a - s \rfloor + 2s - a\right\}$.
\end{theorem}

\addtocontents{toc}{\protect\setcounter{tocdepth}{2}}
\section{Falconer type functions}\label{sec_3}
In this section, we present the proofs for our main results on Falconer type functions. For $y\in \mathbb{R}^n$, denote $\Delta_{\Phi,y}(A):=\left\{\Phi(x,y): x\in A\right\}$. Theorem \ref{thm_main} is part (i) of the following general statement. Recall that $\rho:\mathbb{R}^n\setminus \{0\}\to S^{n-1}, x\mapsto \frac{x}{\Vert x\Vert}$ is the radial projection.
\begin{theorem}\label{thm_main_general}
    Let $1\leq m\leq n,  1\leq p\leq n-1$, and let $U\subset \mathbb{R}^m$, $V\subset \mathbb{R}^p$ be open sets. Let $P:U\to \mathbb{R}^{n}$ and $Q:V\to \mathbb{R}^{n}\setminus\{0\}$ be smooth maps such that $\rank DP$ and $\rank D(\rho\circ Q)$ are maximal everywhere. Let $\Phi:\mathbb{R}^m\times \mathbb{R}^p\to \mathbb{R}$ be a function defined by
    \begin{align*}
        \Phi(x,y)\coloneqq\sum\limits_{k=1}^{n}P_k(x)Q_k(y),\quad \forall (x,y)\in U\times V.
    \end{align*}
    Assume that $A\subset U$ and $B\subset V $ are compact sets. 
    \begin{itemize}
        \item[(i)] If $\dim_{H} A+\dim_{H} B>n$, then there exists $y\in B$ such that $\mathcal{L}^1\left(\Delta_{\Phi,y}(A)\right)>0$.
        \item[(ii)] If $\dim_{H} A+\dim_{H} B>n+1$, then there exists $y\in B$ such that $\Int\left(\Delta_{\Phi,y}(A)\right)\neq \emptyset$.
        \item[(iii)] Let $\max\left\{\dim_{H} A-(n-1),0\right\}\leq u\leq \min\left\{\dim_{H} A,1\right\}$. If $\dim_{H}  B\geq S\left(\dim_{H} A,u\right)$, then there exists $y\in B$ such that $\dim_{H} \Delta_{\Phi,y}(A)\geq u$.
    \end{itemize}
    Here, $\Int(X)$ denotes the interior of the set $X$.
\end{theorem}

\begin{proof}[Proof of Theorem \ref{thm_main_general}]
 For each $\theta\in S^{n-1}$, let $\pi_\theta:\mathbb{R}^n\to \mathbb{R}$ denote the orthogonal projection onto the line in the direction $\theta$. Then for $(x,y)\in U\times V$, one can write
\begin{align*}
   \Phi(x,y)&=\sum\limits_{k=1}^{n}P_k(x)Q_k(y)\\
   &= \Vert Q(y)\Vert\left\langle P(x),\frac{Q(y)}{\Vert Q(y)\Vert}\right\rangle\\
   &=\Vert Q(y)\Vert\pi_{\rho\circ Q(y)}(P(x)).
\end{align*}
For fixed $y\in V$, since $Q(y)\neq 0$, the map $u\mapsto \Vert Q(y)\Vert u$ is bi-Lipschitz. Therefore
\begin{align}\label{eq_bilipschitz}
    \dim_{H} \Delta_{\Phi,y}(A)= \dim_{H} \left(\pi_{\rho\circ Q(y)}P(A)\right),\quad \text{ and }\quad \mathcal{L}^1 (\Delta_{\Phi,y}(A))\Longleftrightarrow\mathcal{L}^1\left(\pi_{\rho\circ Q(y)}P(A)\right).
\end{align}
Similarly, if $\Int\left(\Delta_{\Phi,y}(A)\right)\neq \emptyset$ then $\Int\left(\pi_{\rho\circ Q(y)}P(A)\right)\neq \emptyset$, and vice versa.

By assumptions, $P$ is a smooth immersion, which is a local embedding at every point $x\in U$. For each $x\in U$, let $B_x$ be an open ball such that $P\vert_{B_x}: B_x\to P(B_x)$ is a smooth embedding. The collection $\{B_x\}_{x\in U}$ is an open cover of $U$. Since $A\subset U$ is a compact set, there exists a finite collection $\{B_{x_i}\}_{i=1}^l$ such that $A\subset \bigcup\limits_{i=1}^l B_{x_i}$. By pigeonholing, there exists $x_i$ such that $\dim_{H}  A\cap B_{x_i}=\dim_{H} A$. We then restrict $U$ to $B_{x_i}$, and $A$ to $A\cap B_{x_i}$, keeping the same notation $U$ and $A$ for simplicity. Now the map $P:U\to P(U)$ is a smooth embedding and hence bi-Lipschitz on compact subsets. Consequently, $\dim_{H}  P(A)=\dim_{H}  A$.

Similarly, we can restrict $V$ and $B$ such that $\rho\circ Q: V\to \rho\circ Q(V)$ is a bi-Lipschitz map. This yields that $\dim_{H}  \rho\circ Q(B)=\dim_{H}  B$.
 
We now apply Theorem \ref{thm_projection_1} and Theorem \ref{thm_projection_2} to obtain desired conclusions. More precisely, for part (i), the assumptions $\dim_{H} P(A)+\dim_{H}  \rho\circ Q(B)>n$, and $\dim_{H}  \rho\circ Q(B)\leq n-1$ imply $\dim_{H} P(A)>1$. Hence by Theorem \ref{thm_projection_1} (i), there exists $\theta\in \rho\circ Q(B)$ such that $\mathcal{L}^1(\pi_{\theta}P(A))>0$. Using the identity \eqref{eq_bilipschitz} yields the desired result. 
By the same reasoning, part (ii) follows from Theorem \ref{thm_projection_1} (ii) and part (iii) follows from Theorem \ref{thm_projection_2}.
\end{proof}

When $p=1$, the map $Q$ in Theorem \ref{thm_main_general} parametrizes a curve. The next theorem provides an explicit regularity condition for such curves.

Let $n\geq 2$, we write $(x,y)=(x_1,\dots,x_{n},y)\in \mathbb{R}^{n+1}$. For two nonzero vectors $u,v\in \mathbb{R}^n$, we write $u\parallel v$ if they are parallel, and $u\nparallel v$ otherwise.
\begin{theorem}\label{thm_smoothcurve}
Let $U\subset \mathbb{R}^{n}$, $I\subset \mathbb{R}$ be open sets, let $f: U\to \mathbb{R}$ be a smooth function of the form
\begin{align*}
    f(x_1,\dots,x_{n},y)=\sum\limits_{k=1}^{n}P_k(x)\varphi_k(y), \quad \forall x\in U, y\in I,
\end{align*}
where $P: U\to \mathbb{R}^{n}$ and $\varphi:I\to \mathbb{R}^{n}$ are smooth maps satisfying
\begin{align*}
    \rank DP=n \quad\text{ everywhere},  \quad \varphi'(t)\nparallel \varphi(t), \quad\text{and }\quad  \varphi(t)\neq 0,\,\,\, \forall t\in I.
\end{align*}
Assume that $A\subset U$, $B\subset I$ are compact sets such that $$\dim_{H} (A) + \dim_{H} (B) >n.$$
Then there exists $y\in B$ such that $\mathcal{L}^1\left(f(A,y)\right) > 0 $.
\end{theorem}

\begin{proof}[Proof of Theorem \ref{thm_smoothcurve}]
    We only need to verify that $(\rho\circ \varphi)'(t)\neq 0$ for all $t\in I$. Once this is established, Theorem \ref{thm_main_general} (i) immediately yields the result.

   Define $\psi(t)\coloneqq\rho\circ\varphi(t)=\frac{\varphi(t)}{\Vert\varphi(t)\Vert}$, for $t\in I$. A direct computation gives
   \begin{align*}
       \psi'(t)=\frac{\Vert\varphi(t)\Vert^2\varphi'(t)-(\varphi(t)\cdot\varphi'(t))\varphi(t)}{\Vert\varphi(t)\Vert^3}, \quad \forall t\in I.
   \end{align*}
   Therefore,
   \begin{align*}
       \Vert\psi'(t)\Vert^2
       &=\frac{\Vert\varphi(t)\Vert^2\Vert\varphi'(t)\Vert^2-\vert\varphi(t)\cdot\varphi'(t)\vert^2}{\Vert\varphi(t)\Vert^6},\quad \forall t\in I.
   \end{align*}
   It follows that  $\psi'(t)=0$ precisely when $\vert\varphi(t)\cdot \varphi'(t)\vert=\Vert\varphi(t)\Vert \cdot \Vert\varphi'(t)\Vert$, which  occurs only if $\varphi(t)\parallel \varphi'(t)$. This contradicts the assumption $\varphi(t)\nparallel \varphi'(t)$ for all $t\in I$; therefore $\psi'(t)\neq 0$ for all $t\in I$ as required.
\end{proof}
We can now prove Theorem \ref{thm_polys}.
\begin{proof}[Proof of Theorem \ref{thm_polys}]
    Let $\varphi:\mathbb{R}\to \mathbb{R}^{n}$ be a smooth map defined by
    \begin{align*}
        \varphi(y)\coloneqq\left(y^{r_1},\dots,y^{r_{n}}\right), \quad \forall y\in \mathbb{R}.
    \end{align*}
    Differentiating and noting that  $r_i\neq r_j$, for some $1\leq i<j\leq n$, one can check that  
    \begin{align*}
        \varphi'(y)\nparallel \varphi(y),\quad \text{ for all }y\neq 0.
    \end{align*}
    Applying Theorem \ref{thm_smoothcurve}, we obtain the desired result.
\end{proof}

\section{Distance sets}\label{sec_4}
In this section, we apply the results of the previous section to prove Theorems \ref{thm_distance_hyperplane} and \ref{thm_distance_product}. Our approach is to reduce the study of the distance problem to the framework in Section \ref{sec_3}, treating the distance function  as a special case.

We begin with the case where the set of pins lies in an affine hyperplane.

\subsection{Proof of Theorem \ref{thm_distance_hyperplane}}
Let $n\geq 2$, and let $A$, $B$ be compact sets satisfying the hypotheses of the theorem. Observe that the second part of the theorem follows from the first by taking $B=A\cap W$; hence, it suffices to prove the first part.

Since the distance function is invariant under rotations and translations, we may assume that $B\subset \mathbb{R}^{n-1}\times \{0\}$. Observe that if $A$ is also contained in $\R^{n-1}\times \{0\}$, then the result follows from Falconer's theorem. Therefore, we can assume that $\dim_H\big(A\cap (\R^{n-1}\times (0,\infty)\big)=\dim_HA$, and by abuse of notation, we still refer to $A\cap (\R^{n-1}\times (0,\infty)\big)$ as $A$.

For $x\in \mathbb{R}^n$, we write $x=(x_1,\dots,x_n)=(x',x_n)$. Let $x=(x',x_n)\in A$, $y=(y',0)\in B$, then
    \begin{align*}
        \vert x-y\vert^2
        &=\sum\limits_{i=1}^{n} x_i^2- 2\sum\limits_{i=1}^{n-1} x_iy_i+\sum\limits_{i=1}^{n-1}y_i^2\\
        &=\left\langle (x',\Vert x\Vert^2),(-2y',1)\right\rangle+\Vert y'\Vert^2.
    \end{align*}
    Define $\Phi:\mathbb{R}^{n}\times \mathbb{R}^{n-1}\to \mathbb{R}$ by 
    $$\Phi(x,y')\coloneqq\left\langle \left(x',\Vert x\Vert^2\right),\left(-2y',1\right)\right\rangle+\Vert y'\Vert^2,\quad\quad \text{for } x\in \R^{n}, y'\in \R^{n-1}.$$
    One can check that $P(x)=\left(x',\Vert x\Vert^2\right)$ satisfies $\rank(DP)=n$ away from $\{x_n=0\}$, and $Q(y')=\left(-2y',1\right)\in \R^n\setminus \{0\}$ for all $y'\in \R^{n-1}$. Therefore, $\Phi$ satisfies the hypotheses of Theorem \ref{thm_main_general}. Hence, applying Theorem \ref{thm_main_general} (i), there exists $y=(y',0)\in B$ such that $\mathcal{L}^1\left(\Delta_{\Phi,y'}(A)\right)>0$, which in turn implies $\mathcal{L}^1\left(\Delta_y(A)\right)>0$ as desired.
\qed

\subsection{Proof of Theorem \ref{thm_distance_product}}
In this section, we prove the following general statement, from which Theorem \ref{thm_distance_product} follows ($\theta=1$).

\begin{theorem}\label{thm_distance_product_Fourier_spectrum}
    Let $n\geq 2$ and $0\leq m\leq n-1$. Let $A=A_1\times A_2$ and $B=B_1\times B_2$ be compact sets in $\R^n=\R^{m}\times \R^{n-m}$
    %, where $A_1,B_1\subset \R^m$, and $A_2,B_2\subset \R^{n-m}$
    . 
    Assume that for some $\theta\in [0,1]$,
    \begin{align*}
    \dim_{H} A_1+\dim_{H} B_1+\dim_{F}^\theta A_2+\dim_{F}^{1-\theta}B_2>n,
    \end{align*} 
    Then there exists $y\in B$ such that $\mathcal{L}^1\left(\Delta_y(A)\right)>0$.

    In particular, if for some $\theta\in [0,1]$,
    $$\dim_{H} A_1+\frac{1}{2}\left(\dim_{F}^\theta A_2+\dim_{F}^{1-\theta}A_2\right)>\frac{n}{2},$$
    then there exists $y\in A$ such that $\mathcal{L}^1\left(\Delta(A)\right)>0$.
\end{theorem}

\begin{proof}
Observe that the second part of the theorem follows from the first by taking $A=B$, therefore, it suffices to prove the first part.

Without loss of generality, assume that $m=n-1$ and that $A,B\subset \mathbb{R}^{n-1}\times \R$ are compact sets satisfying the assumptions of Theorem \ref{thm_distance_product_Fourier_spectrum} with $\theta_0\in (0,1)$, namely
\begin{align*}
    \dim_{H} A_1+\dim_{H} B_1+\dim_{F}^{\theta_0} A_2+\dim_{F}^{1-\theta_0}B_2>n.
    \end{align*}
    The boundary cases $\theta_0=0$ or $\theta_0=1$ can be treated similarly. By translation invariance, we may further assume that $B_1\times\{0\}\subset B$ for convenience.
    
First, suppose that $\dim_{F}^{\theta_0} A_2+\dim_{F}^{1-\theta_0}B_2\leq 1$.
    For $x\in \mathbb{R}^n$, we write $x=(x_1,\dots,x_{n-1},x_n)=(x',x_n)$ and similarly for $y\in \mathbb{R}^n$. As in the proof of Theorem \ref{thm_distance_hyperplane}, we compute
    \begin{align*}
        \vert x-y\vert^2&=\sum\limits_{i=1}^{n-1} x_i^2- 2\sum\limits_{i=1}^{n-1} x_iy_i+\sum\limits_{i=1}^{n-1}y_i^2+(x_n-y_n)^2\\
        &=\left\langle\left(x',\sum\limits_{i=1}^{n-1} x_i^2+(x_n-y_n)^2\right),(-2y',1)\right\rangle+\Vert y'\Vert^2.
    \end{align*}
    For $x',y'\in \mathbb{R}^{n-1}, z\in \mathbb{R}$, define
    \begin{align*}
        \Phi\left(x',z,y'\right)\coloneqq\left\langle\left(x',\sum\limits_{i=1}^{n-1} x_i^2+z^2\right),(-2y',1)\right\rangle+\Vert y'\Vert^2.
    \end{align*}
    Then $\Phi$ satisfies the hypotheses of Theorem \ref{thm_main_general}. Let $A'\coloneqq A_1\times (A_2-B_2)\subset \mathbb{R}^{n}$, and $B'\coloneqq B_1$. It suffices to show that
    \begin{align*}
        \dim_{H} A'+\dim_{H} B'>n.
    \end{align*}
    Indeed, according to Theorem \ref{thm_main_general}, there exists $y_0'\in B'$ such that $\Delta_{\Phi,y_0'}(A')$ has positive Lebesgue measure. 
    This yields that for some $y\in B$, $\mathcal{L}^1\left(\Delta_y(A)\right)>0$ as desired.

    We now verify that $\dim_{H} A'+\dim_{H} B'>n$. By the assumptions, there exists $\e>0$ such that
    \begin{align}\label{eq_sum_dim}
\dim_{H} A_1+\dim_{H} B_1+\dim_{F}^{\theta_0}A_2+\dim_{F}^{1-\theta_0}B_2>n+\e.
    \end{align}
    It therefore suffices to show that
    \begin{align}\label{eq_A2_B_2}
        \dim_{H} (A_2-B_2)\geq\dim_{F}^{\theta_0}A_2+\dim_{F}^{1-\theta_0}B_2-\e.
    \end{align}
    Indeed, combining \eqref{eq_sum_dim} with \eqref{eq_A2_B_2}, we obtain
    \begin{align*}
        \dim_{H} A'+\dim_{H} B' &=\dim_{H} A_1+\dim_{H} (A_2-B_2)+\dim_{H} B_1\\
        &\geq \dim_{H} A_1+(\dim_{F}^{\theta_0}A_2+\dim_{F}^{1-\theta_0}B_2-\e)+\dim_{H} B_1\\
        &>n.
    \end{align*}
    The proof of inequality \eqref{eq_A2_B_2} is standard. 
    Let $\alpha < \dim_{F}^{\theta_0}A_2$ and $\beta < \dim_{F}^{1-\theta_0}B_2$ be such that
    \begin{align*}
        \alpha+\beta>\dim_{F}^{\theta_0}A_2+\dim_{F}^{1-\theta_0}B_2-\e.
    \end{align*}
    Let $\mu\in \mathcal{M}(A_2)$ and $\nu\in \mathcal{M}(-B_2)$ be measures such that $\J_{\alpha,\theta_0}(\mu)<\infty$, and $\J_{\beta,1-\theta_0}(\nu)<\infty$. Noting that $\mu\ast \nu\in \mathcal{M}(A_2-B_2)$. By H{\"o}lder's inequality, one has
    \begin{align*}
        \int\vert\widehat{\mu\ast \nu}(\xi)\vert^2\vert\xi\vert^{\alpha+\beta-1}\,d\xi&=\int \vert\widehat{\mu}(\xi)\vert^2\vert\xi\vert^{\alpha-\theta_0}\vert\widehat{\nu}(\xi)\vert^2\vert\xi\vert^{\beta-(1-\theta_0)} \,d\xi\\
        &\leq\left(\int\vert\widehat{\mu}(\xi)\vert^{\frac{2}{\theta_0}}\vert\xi\vert^{\frac{\alpha}{\theta_0} -1}\,d\xi\right)^{\theta_0}\left(\int \vert\widehat{\nu}(\xi)\vert^{\frac{2}{1-\theta_0}}\vert\xi\vert^{\frac{\beta}{1-\theta_0}-1}\,d\xi\right)^{1-\theta_0}\\
        &\leq \J_{\alpha,\theta_0}(\mu)\J_{\beta,1-\theta_0}(\nu)
        <\infty.
    \end{align*}
    Hence, $\dim_{H} (A_2-B_2)\geq \alpha+\beta>\dim_{F}^{\theta_0}A_2+\dim_{F}^{1-\theta_0}B_2-\e$. This completes the proof of Theorem \ref{thm_distance_product} in the case $\dim_{F}^{\theta_0}A_2+\dim_{F}^{1-\theta_0}B_2\leq 1$.

    Now we assume that $\dim_{F}^{\theta_0}A_2+\dim_{F}^{1-\theta_0}B_2>1$. In this case, by the same argument, we can choose $0<\alpha<\dim_{F}^{\theta_0}A_2$, $0<\beta<\dim_{F}^{1-\theta_0}B_2$ such that
    \begin{align*}
        \dim_{F}^{\theta_0}A_2+\dim_{F}^{1-\theta_0}B_2>\alpha+\beta=1.
    \end{align*}
    Let $\mu$ and $\nu$ be as above, so that
    \begin{align*}
        \int\vert\widehat{\mu\ast \nu}(\xi)\vert^2\,d\xi=\int\vert\widehat{\mu\ast \nu}(\xi)\vert^2\vert\xi\vert^{\alpha+\beta-1}\,d\xi<\infty.
    \end{align*}
    Thus, $\mu\ast \nu\in L^2$, which implies that $\mathcal{L}^1\left(A_2-B_2\right)>0$. By the definition of $\Phi\left(x',z,y'\right)$ above, one can check that for fixed $x'\in A_1$, $y'\in B_1$, the set
    \begin{align*}
        \Phi\left(x',A_2-B_2,y'\right)=\left\{\Phi(x',z,y'): z\in A_2-B_2\right\}
    \end{align*}
    has positive Lebesgue measure. 
    Therefore, for some $y\in B$, we have $\mathcal{L}^1\left(\Delta_y(A)\right)>0$. This completes the proof of the theorem.
\end{proof}

\section{Examples}\label{sec_5}
In this section, we give some examples that illustrate that the assumptions in our main results are necessary.

\begin{example}\label{example_5.1}
    Let $n\geq 2$, and let $f\in \mathbb{R}[x_1,\dots,x_{n},y]$ be a polynomial of the form
\begin{align*}
    f(x,y)= x_1y+\dots +x_{n-1}y+x_{n}y^2.
\end{align*}
One readily checks that $f$ satisfies the hypotheses of Theorem \ref{thm_polys}. 
\begin{itemize}
    \item[(i)] Let $A\subset \mathbb{R}^{n}$ be a compact set of dimension $\dim_{H} A=n$ and $B=\{0\}$. Then $\dim_{H} A+\dim_{H} B=n$, but
\begin{align*}
    \Delta_f(A,B)=\left\{ f(x,y): x\in A,y\in B\right\}=\{0\}.
\end{align*}
\item[(ii)] Let $B=\{y\}\subset \mathbb{R}$, with $y\neq 0$.
Set $e=(y,\dots,y,y^2)\in \mathbb{R}^{n}$, let $A_1=e^\perp$ be the hyperplane orthogonal to $e$, and define $l_e=\left\{te: t\in \mathbb{R}\right\}$. Choose $A_2\subset l_e$ compact with $\dim_{H} A_2=1$ but $\mathcal{L}^1(A_2)=0$. Then, with $A=A_1\times A_2$, we have $\dim_{H} A=\dim_{H} A_1+\dim_{H} A_2=n$, yet $\mathcal{L}^1\left(\Delta_f(A,B)\right)=0$.
\end{itemize}
Hence condition that $\dim_{H} A+\dim_{H} B>n$ in Theorem \ref{thm_polys} is sharp and cannot be relaxed.
\end{example}

\begin{example}\label{example_5.2}
 Let $n\geq 2$, and let $g\in \mathbb{R}[x_1,\dots, x_{n},y]$ be a polynomial defined by
    \begin{align*}
        g(x,y)\coloneq x_1y+\dots+x_{n}y=(x_1,\dots,x_{n})\cdot (y,\dots,y),\quad \forall x\in\mathbb{R}^{n},y>0.
    \end{align*}
    Then $g$ satisfies the assumptions of Theorem \ref{thm_polys} except that $r_k=1$ for all $1\leq k\leq n$. Let $e=(1,\dots,1)\in S^{n-1}$, and let $A=e^\perp$, the hyper plane orthogonal to $e$. Take $B\subset [1,2]$ to be a compact set with $\dim_HB=1$ but $\mathcal{L}^1(B)=0$. Then
    \begin{align*}
        \dim_{H} A=n-1,\quad \dim_{H} B=1,
    \end{align*}
    yet $\mathcal{L}^1(\Delta_g(A,B))=0$. Therefore the condition $r_i\neq r_j$ for some $1\leq i\leq j\leq n$ is necessary.
\end{example}

The following sharpness example for Theorem \ref{thm_distance_hyperplane} follows from the sharpness example for exceptional sets estimates for projections to lines \cite[Example 5.3]{Mattila2015}.

\begin{example}\label{example_pinned_distance}
    Let $n\geq 2$, $0<s\leq n$. Let $0<\delta<1$. Choose a rapidly increasing sequence $(q_k)_{k\in \N}$ of positive integers, say $q_{k+1}>q_k^k$ for all $k\in \N$. Denote by $\Vert x\Vert$ the distance of the real number $x$ to the nearest integer and define sets
\begin{align*}
    A'&\coloneq\left\{ x\in [0,1]^n: \Vert q_kx_i\Vert\leq q_k^{1-n/s} \text{ for all } k\in \N, i=1,\dots, n\right\},\\
    B'_\delta&\coloneq\bigg\{(y_1,\dots, y_{n-1})\in \mathbb{R}^{n-1}: \text{ for infinitely many } k\in \N , \exists\,m_k\in \N\cap \left[1,q_k^{(1-\delta)(n-s)/s}\right] \\
    &\hspace{5cm} \text{ such that }\Vert m_ky_i\Vert\leq m_kq_k^{-n/s} \text{ for all } i=1,\dots, n-1\bigg\}.
\end{align*}
Then $\dim_{H} A'=s$ and $\dim_{H}  B'_\delta=(1-\delta)(n-s)$, see \cite[Jarn\'ik's Theorem 10.3]{Falconerbook14}. 

For all $y\in B'_\delta$, we have that
\begin{align*}
    \mathcal{L}^1(\tilde{\pi}_{\theta_y}(A'))=0,
\end{align*}
where $\theta_y=(y,1)\in \mathbb{R}^n$, and $\tilde{\pi}_{\theta_y}(x)=\theta_y\cdot x$, for $x\in \mathbb{R}^n$. It follows that
\begin{align*}
    \dim_{H} \left\{ \theta\in S^{n-1}: \mathcal{H}^1(\tilde{\pi}_\theta (A'))=0\right\}\geq (1-\delta)(n-s).
\end{align*}
Letting $\delta\to 0$, we conclude that there exists a set $B'\subset \mathbb{R}^{n-1}$, $\dim_{H} B'=n-s$, and $\mathcal{L}^{1}\left(\pi_{(-2y',1)}(A')\right)=0$ for all $y'\in B'$, where $\pi_{(-2y',1)}(x)= (-2y',1)\cdot (x',x_n)$.

For $x\in \mathbb{R}^n$, we write $x=(x_1,\dots,x_n)=(x',x_n)$, and define a smooth map
\begin{align*}
    F: \mathbb{R}^n\to \mathbb{R}^n, \quad (x',x_n)\mapsto\left(x', \sum\limits_{i=1}^{n-1}x_i^2+x_n^2\right).
\end{align*}
Then $F$ is bi-Lipschitz on the compact set $U=\left[0,1/n\right]^{n-1}\times\left[1/4,1/2\right]$. Set $A\coloneqq F^{-1}(A')$, then $\dim_{H} A=\dim_{H} A'=s$. Define
\begin{align*}
    \Phi:\mathbb{R}^n\times \mathbb{R}^{n-1},\quad \Phi\left(x',x_n,y'\right)\coloneq\left(x', \sum\limits_{i=1}^{n-1}x_i^2+x_n^2\right)\cdot (-2y',1).
\end{align*}
One can verify that for all $y'\in B$, $\mathcal{L}^1\left(\Delta_{\Phi,y'}(A)\right)=0$, and that $\dim_{H} A+\dim_{H} B=n$.
\end{example}
Another sharpness example for Theorem \ref{thm_distance_hyperplane} is the following.
\begin{example}
\begin{itemize}
    \item[]
    \item[(i)] Let $n\geq 2$. For each $t\in \mathbb{R}_{>0}$, denote $S_{t}^{n-1}$ as the $(n-1)$-dimensional sphere of radius $t$. Let $T\subset \mathbb{R}_{>0}$ be a compact set with $\dim_{H}T =1$ but $\mathcal{L}^1(T)=0$. Let $A\subset \mathbb{R}^n$ be the union of all spheres with radii in $T$, namely,  $$A=\bigcup\limits_{t\in T}S_t^{n-1}.$$ Choose $B=\{0\}$, then we have $\dim_{H} A+\dim_{H}  B=n$. One can see that
    \begin{align*}
        \Delta(A,B)=\left\{ \vert x-y\vert: x\in A, y\in B\right\}=T,
    \end{align*}
    consequently $\mathcal{L}^1\left(\Delta(A,B)\right)=0$.
    \item[(ii)] Let $n\geq 2$, we claim that for any $\e>0$, there exist compact sets $A, B\subset \mathbb{R}^n$ such that $B$ lies in a hyperplane, $\dim_{H} A+\dim_{H} B=n-\e$, and
    \begin{align*}
        \mathcal{L}^1\left(\Delta(A,B)\right)=0.
    \end{align*}
    Indeed, fix $\e>0$ and let $C\subset [1/2,1]$ be a compact set with $\dim_{H} C=1-\e/2$ and $\dim_{H} (C+C)=1-\e/2$ (see \cite{Falconer85,Schmeling2010}). Define
    \begin{align*}
        C'=\left\{\sqrt{t}: t\in C\right\},\quad A=\left\{0\right\}\times C'\subset \mathbb{R}^{n-1}\times \mathbb{R},
    \end{align*}
    so that $A$ lies on the line orthogonal to the hyperplane $V=\mathbb{R}^{n-1}\times \{0\}$.
    For each $t\in C'$, let $S_{t}^{n-2}$ denote the $(n-2)$-dimensional sphere of radius $t$ in $V$, and set
    \begin{align*}
        B=\bigg(\bigcup\limits_{t\in C'} S_{t}^{n-2}\bigg)\times \{0\}\subset \mathbb{R}^{n-1}\times \mathbb{R},
    \end{align*}
    which lies in the hyperplane $V$.

    Then
    \begin{align*}
        \dim_{H} A+\dim_{H} B=1-\e/2+(n-2)+(1-\e/2)=n-\e.
    \end{align*}
  Furthermore,
    \begin{align*}
        \Delta(A,B)^2&=\left\{ \vert x-y\vert^2: x\in A,y\in B\right\}\\
        &=\left\{ t^2+u^2: t,u\in C'\right\}\\
        &=\left\{ t+u: t,u\in C\right\}\\
        &=C+C.
    \end{align*}
    Since $\dim_{H} (C+C)=1-\e/2$, we conclude that $\mathcal{L}^1\left(\Delta(A,B)\right)=0$.
\end{itemize}
\end{example}

\addtocontents{toc}{\protect\setcounter{tocdepth}{0}}
\section*{Acknowledgments}
This paper is part of the author’s Ph.D. thesis at the University of Rochester. The author would like to thank Alex Iosevich and Allan Greenleaf for their valuable insights, helpful discussions, and support. The author is also very grateful for the referee's careful reading of the manuscript and valuable suggestions which improved the paper.

\section*{Data availability}
There are no data associated with this manuscript.

\section*{Conflict of interest}
On behalf of all authors, the corresponding author states that there is no Conflict of interest.

\addtocontents{toc}{\protect\setcounter{tocdepth}{1}}
\bibliographystyle{abbrv}
\bibliography{ref} 

\end{document}